\documentclass[a4paper,11pt]{amsart}

\usepackage{color}
\definecolor{refs}{rgb}{0.7,0,0}
\definecolor{ext}{RGB}{112,112,112}
\definecolor{cite}{RGB}{034,113,179}
\usepackage[colorlinks=true,citecolor=cite,linkcolor=refs,urlcolor=ext,backref=page]{hyperref}

\usepackage{amssymb}

\renewcommand{\i}{\mathrm{i}}
\renewcommand{\d}{\mathrm{d}}
\newcommand{\R}{\mathbb{R}}
\newcommand{\Z}{\mathbb{Z}}
\newcommand{\C}{\mathbb{C}}
\newcommand{\tr}{\operatorname{tr}}
\newcommand{\ov}{\overline}
\newcommand{\e}{\mathrm{e}}
\newcommand{\eps}{\varepsilon}
\newcommand{\levcon}{\varphi}

\newcommand{\rot}{z\rtimes r\e^{\i\phi}}

\newtheorem{thm}{Theorem}
\newtheorem{cor}{Corollary}
\newtheorem{prp}{Proposition}
\newtheorem{lemma}{Lemma}
\theoremstyle{remark}
\newtheorem{rmk}{Remark}
\newtheorem{ex}{Example}

\subjclass[2010]{Primary 53A20; Secondary 53C24, 53C28}

\keywords{projective structures, conformal connections, geodesic rigidity, twistor space}

\numberwithin{equation}{section}

\date{May 9, 2015.}
\thanks{Research for this article was carried out while the author was visiting the Mathematical Institute at the University of Oxford as a postdoctoral fellow of the Swiss NSF, PA00P2\_142053. The author would like to thank the Mathematical Institute for its hospitality.}

\title[Conformal connections on surfaces]{Geodesic rigidity of conformal connections on surfaces}
\author[T.~Mettler]{Thomas Mettler}

\begin{document}

\begin{abstract}
We show that a conformal connection on a closed oriented surface $\Sigma$ of negative Euler characteristic preserves precisely one conformal structure and is furthermore uniquely determined by its unparametrised geodesics. As a corollary it follows that the unparametrised geodesics of a Riemannian metric on $\Sigma$ determine the metric up to constant rescaling. It is also shown that every conformal connection on the $2$-sphere lies in a complex $5$-manifold of conformal connections, all of which share the same unparametrised geodesics.   
\end{abstract}
\maketitle

\section{Introduction}
A~\textit{projective structure} $\mathfrak{p}$ on a surface $\Sigma$ is an equivalence class of affine torsion-free connections on $\Sigma$ where two connections are declared to be projectively equivalent if they share the same geodesics up to parametrisation. A surface equipped with a projective structure will be called a~\textit{projective surface}. In ~\cite{MR3144212} it was shown that an oriented projective surface $(\Sigma,\mathfrak{p})$ defines a complex surface $Z$ together with a projection to $\Sigma$ whose fibres are holomorphically embedded disks. Moreover, a conformal connection in the projective equivalence class corresponds to a section whose image is a holomorphic curve in $Z$. Locally such sections always exist and hence every affine torsion-free connection on a surface is locally projectively equivalent to a conformal connection. The problem of characterising the affine torsion-free connections on surfaces that are locally projectively equivalent to a Levi-Civita connection was recently solved in~\cite{MR2581355}.   

Here we show that if a closed holomorphic curve $D\subset Z$ is the image of a section of $Z \to \Sigma$, then its normal bundle $N \to D$ has degree twice the Euler characteristic of $\Sigma$. This is achieved by observing that the projective structure on $\Sigma$ canonically equips the co-normal bundle of $D$ with a Hermitian bundle metric whose Chern connection can be computed explicitly. Using the fact that the normal bundle $N \to D$ has degree $2\chi(\Sigma)$ and that the bundle $Z \to \Sigma$ has a contractible fibre, we prove that on a closed surface $\Sigma$ with $\chi(\Sigma)<0$ there is at most one section of $Z \to \Sigma$ whose image is a holomorphic curve. It follows that a conformal connection on $\Sigma$ preserves precisely one conformal structure and is furthermore uniquely determined by its unparametrised geodesics. In particular, as a corollary one obtains that the unparametrised geodesics of a Riemannian metric on $\Sigma$ determine the metric up to constant rescaling, a result previously proved in~\cite{MR1796527}. 

In the case where $\Sigma$ is the $2$-sphere, it follows that the normal bundle of a holomorphic curve $D\simeq \mathbb{CP}^1\subset Z$, arising as the image of a section of $Z \to S^2$, is isomorphic to $\mathcal{O}(4)$. Consequently, Kodaira's deformation theorem can be applied to show that every conformal connection on $S^2$ lies in a complex $5$-manifold of conformal connections, all of which share the same unparametrised geodesics. 

\subsection*{Acknowledgements}
This paper would not have come into existence without several very helpful discussions with Nigel Hitchin. I would like to warmly thank him here. I also wish to thank Vladimir Matveev for references and the anonymous referee for her/his careful reading and useful suggestions.  

\section{Projective structures and conformal connections}

In this section we assemble the essential facts about projective structures on surfaces and conformal connections that will be used during the proof of the main result. Here and throughout the article -- unless stated otherwise -- all manifolds are assumed to be connected and smoothness,
i.e. infinite differentiability, is assumed. Also, we let $\R^n$ denote the space of column vectors of height $n$ with real entries and $\R_n$ the space of row vectors of length $n$ with real entries so that matrix multiplication $\R_n \times \R^n \to \R$ is a non-degenerate pairing identifying $\R_n$ with the dual vector space of $\R^n$. Finally, we adhere to the convention of summing over repeated indices.  

\subsection{Projective structures}

Recall that the space $\mathfrak{A}(\Sigma)$ of affine tor\-sion-free connections on a  surface $\Sigma$ is an affine space modelled on the space of sections of the real vector bundle $V=S^2(T^*\Sigma)\otimes T\Sigma$.\footnote{As usual, by an affine torsion-free connection on $\Sigma$ we mean a torsion-free connection on $T\Sigma$.} We have a canonical trace mapping $\mathrm{tr} : \Gamma(V) \to \Omega^1(\Sigma)$ as well as an inclusion
$$
\iota : \Omega^1(\Sigma) \to \Gamma(V), \quad \alpha \mapsto \alpha \otimes \mathrm{Id}+\mathrm{Id}\otimes \alpha, 
$$
where we define 
$$
\left(\alpha\otimes \mathrm{Id}\right)(v)w=\alpha(v)w \quad \text{and} \quad\left(\mathrm{Id}\otimes \alpha\right)(v)w=\alpha(w)v,  
$$
for all $v,w \in T\Sigma$. Consequently, the bundle $V$ decomposes as $V=V_0\oplus T^*\Sigma$ where $V_0$ denotes the trace-free part of $V$. The projection $\Gamma(V) \to \Gamma(V_0)$ is given by 
$$
\phi \mapsto \phi_0=\phi-\frac{1}{3}\iota\left(\tr \phi\right).
$$
Weyl~\cite{zbMATH02603060} observed that two affine torsion-free connections $\nabla$ and $\nabla^{\prime}$ on $\Sigma$ are projectively equivalent if and only if their difference is pure trace 
\begin{equation}\label{weylchar}
(\nabla-\nabla^{\prime})_0=0.
\end{equation}
We will denote the space of projective structures on $\Sigma$ by $\mathfrak{P}(\Sigma)$. From~\eqref{weylchar} we see that $\mathfrak{P}(\Sigma)$ is an affine space modelled on the  space of smooth sections of $V_0\simeq S^3(T^*\Sigma)\otimes \Lambda^2(T\Sigma)$.  

Cartan~\cite{MR1504846} (see~\cite{MR0159284} for a modern exposition) associates to an oriented projective surface $(\Sigma,\mathfrak{p})$ a~\textit{Cartan geometry} of type $(\mathrm{SL}(3,\R),G)$, which consists of a principal right $G$-bundle $\pi : B \to \Sigma$ together with a~\textit{Cartan connection} $\theta \in \Omega^1(B,\mathfrak{sl}(3,\R))$. The group $G\simeq \R^2\rtimes \mathrm{GL}^+(2,\R)\subset \mathrm{SL}(3,\R)$ consists of matrices of the form
$$
b\rtimes a=\left(\begin{array}{cc} (\det a)^{-1} & b\\ 0 & a\end{array}\right), 
$$
where $a \in \mathrm{GL}^+(2,\R)$ and $b^t \in \R^2$. The Cartan connection $\theta$ is an $\mathfrak{sl}(3,\R)$-valued $1$-form on $B$ which is equivariant with respect to the $G$-right action, maps every fundamental vector field $X_v$ on $B$ to its generator $v \in \mathfrak{g}$, and restricts to be an isomorphism on each tangent space of $B$. Furthermore, the Cartan geometry $(\pi : B \to \Sigma, \theta)$ has the following properties:
\begin{itemize}
\item[(i)] Write $\theta=(\theta^{\mu}_{\nu})_{\mu,\nu=0..2}$. Let $X$ be a vector field on $B$ satisfying $\theta^i_0(X)=c^i$, $\theta^i_j(X)=0$ and $\theta^0_j(X)=0$ for real constants $(c^1,c^2)\neq (0,0)$, where $i,j=1,2$. Then every integral curve of $X$ projects to $\Sigma$ to yield a geodesic of $\mathfrak{p}$ and conversely every geodesic of $\mathfrak{p}$ arises in this way;  
\item[(ii)] an orientation compatible volume form on $\Sigma$ pulls-back to $B$ to become a positive multiple of $\theta^1_0\wedge\theta^2_0$;
\item[(iii)] there exist real-valued functions $W_1,W_2$ on $B$ such that
\begin{equation}\label{struceqprojsurf}
\d \theta+\theta\wedge\theta=\left(\begin{array}{ccc} 0 & W_1 \theta^1_0\wedge\theta^2_0 & W_2 \theta^1_0\wedge\theta^2_0\\ 0 & 0 & 0\\ 0 & 0 &0 \end{array}\right).
\end{equation}
\end{itemize}
The fibre of $B$ at a point $p \in \Sigma$ consists of the $2$-jets of orientation preserving  local diffeomorphisms $\varphi$ with source $0\in \R^2$ and target $p$, so that $\varphi^{-1}$ maps the geodesics of $\mathfrak{p}$ passing through $p$ to curves in $\R^2$ having vanishing curvature at $0$. The structure group $G$ consists of the $2$-jets of orientation preserving fractional-linear transformations with source and target $0 \in \R^2$.  Explicitly, the identification between the matrix Lie group $G$ and the Lie group of such $2$-jets is given by $b\rtimes a \mapsto j^2_0f_{a,b}$ where 
\begin{equation}\label{frac2jets}
f_{a,b} : x \mapsto \frac{(\det a) a\cdot x}{1+(\det a)b\cdot x}
\end{equation}
and $\cdot$ denotes usual matrix multiplication. The group $G$ acts on $B$ from the right by pre-composition, that is, 
\begin{equation}\label{rightaction}
j^2_0\varphi\cdot j^2_0f_{a,b}=j^2_0(\varphi \circ f_{a,b}).
\end{equation}

\begin{rmk}
Cartan's construction is unique in the following sense: If $(B^{\prime}\to \Sigma,\theta^{\prime})$ is another Cartan geometry of type $(\mathrm{SL}(3,\R),G)$ satisfying the properties (i),(ii),(iii), then there exists a $G$-bundle isomorphism $\psi : B \to B^{\prime}$ so that $\psi^*\theta^{\prime}=\theta$. 
\end{rmk}

\begin{ex}\label{flatex}
The Cartan geometry $(\pi : \mathrm{SL}(3,\R)\to \mathbb{S}^2,\theta)$, where $\theta$ denotes the Maurer-Cartan form of $\mathrm{SL}(3,\R)$ and 
$$
\pi : \mathrm{SL}(3,\R) \to \mathrm{SL}(3,\R)/G\simeq \mathbb{S}^2=\left(\R^3\setminus\left\{0\right\}\right)/\R^+
$$
the quotient projection, defines an orientation and projective structure $\mathfrak{p}_0$ on the projective $2$-sphere $\mathbb{S}^2$. The geodesics of $\mathfrak{p}_0$ are the~\textit{great circles} $\mathbb{S}^1\subset \mathbb{S}^2$, that is, subspaces of the form $E\cap \mathbb{S}^2$, where $E\subset\R^3$ is a linear $2$-plane. The group $\mathrm{SL}(3,\R)$ acts on $\mathbb{S}^2$ from the left via the natural left action on $\R^3$ by matrix multiplication and this action preserves both the orientation and projective structure $\mathfrak{p}_0$ on $\mathbb{S}^2$. The unparametrised geodesics of the Riemannian metric $g$ on $\mathbb{S}^2$ obtained from the natural identification $\mathbb{S}^2\simeq S^2$, where $S^2\subset \R^3$ denotes the unit sphere in Euclidean $3$-space, are the great circles. In particular, for every $\psi \in \mathrm{SL}(3,\R)$, the geodesics of the Riemannian metric $\psi^*g$ on $\mathbb{S}^2$ are the great circles as well, hence the space of Riemannian metrics on the $2$-sphere having the great circles as their geodesics contains -- and is in fact equal to -- the real $5$-dimensional homogeneous space $\mathrm{SL}(3,\R)/\mathrm{SO}(3)$.       
\end{ex}

\begin{ex}\label{prefconn}
Here we show how to construct Cartan's bundle from a given affine torsion-free connection $\nabla$ on an oriented surface $\Sigma$. The reader may want to consult~\cite{MR0159284} for additional details of this construction. Let $\upsilon : F^+ \to \Sigma$ denote the bundle of positively oriented~\textit{coframes} of $\Sigma$, that is, the fibre of $F^+$ at $p \in \Sigma$ consists of the linear isomorphisms $u : T_p\Sigma \to \R^2$ which are orientation preserving with respect to the given orientation on $\Sigma$ and the standard orientation on $\R^2$. The group $\mathrm{GL}^+(2,\R)$ acts transitively from the right on each $\upsilon$-fibre by the rule $u\cdot a=R_a(u)=a^{-1} \circ u$ for all $a \in \mathrm{GL}^+(2,\R)$.  This right action makes $F^+$ into a principal right $\mathrm{GL}^+(2,\R)$-bundle over $\Sigma$. Recall that there is a tautological $\R^2$-valued $1$-form $\eta=(\eta^i)$ on $F^+$ defined by
$$
\eta(v)=u(\upsilon^{\prime}(v)), \quad \text{for}\;v \in T_uF^+. 
$$
The form $\eta$ satisfies the equivariance property $(R_a)^*\eta=a^{-1}\eta$ for all $a \in \mathrm{GL}^+(2,\R$). 

Let now $\zeta=(\zeta^i_j) \in \Omega^1(F^+,\mathfrak{gl}(2,\R))$ be the connection form of an affine torsion-free connection $\nabla$ on $\Sigma$. We have the structure equations
$$
\aligned
\d\eta^i&=-\zeta^i_j\wedge\eta^j,\\
\d\zeta^i_j&=-\zeta^i_k\wedge\zeta^k_j+\frac{1}{2}R^i_{jkl}\eta^k\wedge\eta^l
\endaligned
$$
for real-valued curvature functions $R^i_{jkl}$ on $F^+$. As usual, we decompose the curvature functions $R^i_{jkl}$ into irreducible pieces, thus writing\footnote{We define $\eps_{ij}=-\eps_{ji}$ with $\eps_{12}=1$.}
$$
R^i_{jkl}=R_{jl}\delta^i_k-R_{jk}\delta^i_l+R \eps_{kl}\delta^i_j
$$
for unique real-valued functions $R_{ij}=R_{ji}$ and $R$ on $F^+$. Contracting over $i,k$ we get
$$
R^k_{jkl}=2R_{jl}-R_{jl}+R\eps_{jl}=R_{jl}+R\eps_{jl}.
$$ 
Consequently, denoting by $\mathrm{Ric}^{\pm}(\nabla)$ the symmetric and anti-symmetric part of the Ricci tensor of $\nabla$, we obtain
$$
\upsilon^*\left(\mathrm{Ric}^+(\nabla)\right)=R_{ij}\eta^i\otimes \eta^j \quad \text{and}\quad  \upsilon^*\left(\mathrm{Ric}^-(\nabla)\right)=R\eps_{ij}\eta^i\otimes \eta^j. 
$$
In two dimensions, the (projective)~\textit{Schouten tensor} of $\nabla$ is defined as $\mathrm{Sch}(\nabla)=\mathrm{Ric}^+(\nabla)-\frac{1}{3}\mathrm{Ric}^{-}(\nabla)$, so that writing 
$$
\upsilon^*\left(\mathrm{Sch}(\nabla)\right)=S_{ij}\eta^i\otimes \eta^j,
$$
we have
$$
S=(S_{ij})=\begin{pmatrix}R_{11} & R_{12}-\frac{1}{3}R \\ R_{12}+\frac{1}{3}R & R_{22}\end{pmatrix}.
$$

We now define a right $G$-action on $F^+\times \R_2$ by the rule
\begin{equation}\label{gaction}
(u,\xi)\cdot (b\rtimes a)=\left(\det a^{-1} a^{-1} \circ u,\xi a \det a + b\det a\right),
\end{equation}
for all $b\rtimes a \in G$. Here $\xi$ denotes the projection onto the second factor of $F^+\times \R_2$. Let $\pi : F^+\times \R_2 \to \Sigma$ denote the basepoint projection of the first factor. The $G$-action~\eqref{gaction} turns $\pi : F\times \R_2 \to \Sigma$ into a principal right $G$-bundle over $\Sigma$. 
On $F^+\times \R_2$ we define an $\mathfrak{sl}(3,\R)$-valued $1$-form
\begin{equation}\label{cartanconnpref}
\theta=\left(\begin{array}{cc} -\frac{1}{3}\tr \zeta-\xi \eta & \d \xi-\xi\zeta-S^t\eta-\xi\eta\xi\\  \eta & \zeta-\frac{1}{3}\mathrm{I}\tr \zeta+\eta\xi  \end{array}\right).
\end{equation}
Then $(\pi : F^+ \times \R_2 \to \Sigma,\theta)$ is a Cartan geometry of type $(\mathrm{SL}(3,\R),G)$ satisfying the properties (i),(ii) and (iii) for the projective structure defined by $\nabla$. It follows from the uniqueness part of Cartan's construction that $(\pi : F^+ \times \R_2\to\Sigma,\theta)$ is isomorphic to Cartan's bundle. 
\end{ex}

\begin{rmk}
Note that Example~\ref{prefconn} shows that the quotient of Cartan's bundle by the normal subgroup $\R^2\rtimes\left\{\mathrm{Id}\right\}\subset G$ is isomorphic to the principal right $\rm GL^+(2,\R)$-bundle of positively oriented coframes $\upsilon : F^+\to \Sigma$. 
\end{rmk}
     
\subsection{Conformal connections}\label{confconsec}

Recall that an affine torsion-free connection $\nabla$ on $\Sigma$ is called a \textit{Weyl connection} or \textit{conformal connection} if $\nabla$ preserves a conformal structure $[g]$ on $\Sigma$. A torsion-free connection $\nabla$ is $[g]$-conformal if for some -- and hence any -- Riemannian metric $g$ defining $[g]$ there exists a $1$-form $\beta\in \Omega^1(\Sigma)$ such that
\begin{equation}\label{charconfcon}
\nabla g = 2 \beta \otimes g. 
\end{equation}
 Conversely, given a pair $(g,\beta)$ on $\Sigma$, it follows from Koszul's identity that there exists a unique affine torsion-free connection $\nabla$ which satisfies~\eqref{charconfcon}, namely
\begin{equation}\label{confconndown}
{}^{(g,\beta)}\nabla={}^g\nabla+g\otimes \beta^{\sharp}-\iota(\beta),
\end{equation}
where ${}^g\nabla$ denotes the Levi-Civita connection of $g$ and $\beta^{\sharp}$ the $g$-dual vector field to $\beta$. For a smooth real-valued function $u$ on $\Sigma$ we have
\begin{equation}\label{levicivitaconfresc}
{}^{\exp(2u)g}\nabla={}^g\nabla-g\otimes{}^g\nabla u+\iota(\d u), 
\end{equation}
and hence
\begin{equation}\label{confresc}
{}^{(\exp(2u)g,\beta+\d u)}\nabla={}^{(g,\beta)}\nabla.
\end{equation}
Fixing a Riemannian metric $g$ defining $[g]$ identifies the space of $[g]$-conformal connections with the space of $1$-forms on $\Sigma$. It follows that the space of $[g]$-conformal connections is an affine space modelled on $\Omega^1(\Sigma)$. A conformal structure $[g]$ together with a choice of a particular $[g]$-conformal connection $\nabla$ is called a~\textit{Weyl structure}. We will denote the space of Weyl structures on $\Sigma$ by $\mathfrak{W}(\Sigma)$. Furthermore, a Weyl structure $([g],\nabla)$ is called~\textit{exact} if $\nabla$ is the Levi-Civita connection of a Riemannian metric $g$ defining $[g]$. From~\eqref{levicivitaconfresc} we see that the space of exact Weyl structures on $\Sigma$ is in one-to-one correspondence with the space of Riemannian metrics on $\Sigma$ modulo constant rescaling.  
   
Let now $\Sigma$ be oriented (pass to the orientable double cover in case $\Sigma$ is not orientable) and fix a Riemannian metric $g$ and a $1$-form $\beta$ on $\Sigma$. On the bundle $\upsilon : F^+\to \Sigma$ of positively oriented coframes of $\Sigma$ there exist unique real-valued functions $g_{ij}=g_{ji}$ and $b_i$ such that
\begin{equation}\label{someid99}
\upsilon^*g=g_{ij}\eta^i\otimes \eta^j \quad \text{and} \quad \upsilon^*\beta=b_i\eta^i. 
\end{equation}
The $\R_2$-valued function $b=(b_i)$ satisfies the equivariance property
\begin{equation}\label{equivform}
(R_a)^*b=ba. 
\end{equation}
The~\textit{Levi-Civita connection form} of $g$ is the unique $\mathfrak{gl}(2,\R)$-valued connection $1$-form $\varphi=(\varphi^i_j)$ satisfying
$$
\d \eta^i=-\varphi^i_j\wedge\eta^j\quad \text{and}\quad \d g_{ij}=g_{kj}\varphi^k_i+g_{ik}\varphi^k_j. 
$$
The exterior derivative of $\varphi$ can be expressed as
$$
\d \varphi^i_j=-\varphi^i_k\wedge\varphi^k_j+g_{jk}K\eta^i\wedge\eta^k, 
$$
where the real-valued function $K$ is constant on the $\upsilon$-fibres and hence can be regarded as a function on $\Sigma$ which is the~\textit{Gauss-curvature} of $g$. Infinitesimally,~\eqref{equivform} translates to the existence of real-valued functions $b_{ij}$ on $F^+$ satisfying
$$
\d b_i=b_j\varphi^j_i+b_{ij}\eta^j. 
$$
From~\eqref{confconndown} we see that the connection form $\zeta=(\zeta^i_j)$ of the $[g]$-conformal connection ${}^{(g,\beta)}\nabla$ can be written as
$$
\zeta^i_j=\varphi^i_j+\left(b_kg^{ki}g_{jl}-\delta^i_jb_l-\delta^i_lb_j\right)\eta^l,
$$
where the functions $g^{ij}=g^{ji}$ satisfy $g^{ik}g_{kj}=\delta^i_j$. It follows with the equivariance property of $\eta$ and~\eqref{someid99} that the equations $g_{11}\equiv g_{22}\equiv 1$ and $g_{12}\equiv 0$ define a reduction $\lambda : F^+_g \to \Sigma$ of $\upsilon :  F^+ \to \Sigma$ with structure group $\mathrm{SO}(2)$ which consists of the positively oriented coframes that are also $g$-orthonormal. On $F^+_g$ we obtain
$$
\aligned
0&=\d g_{11}=2(g_{11}\varphi^1_1+g_{12}\varphi^2_1)=2\varphi^1_1,\\ 
0&=\d g_{22}=2(g_{21}\varphi^1_2+g_{22}\varphi^2_2)=2\varphi^2_2,\\
0&=\d g_{12}=g_{11}\varphi^1_2+g_{12}\varphi^2_2+g_{12}\varphi^1_1+g_{22}\varphi^2_1=\varphi^1_2+\varphi^2_1. 
\endaligned
$$
Therefore, writing $\varphi:=\varphi^2_1$ we have the following structure equations on $F^+_g$
$$
\aligned
\d \eta^1&=-\eta^2\wedge\levcon,\\
\d \eta^2&=\phantom{-}\eta^1\wedge\levcon,\\
\d \levcon&=-K\eta^1\wedge\eta^2,\\
\d b_i&=b_{ij}\eta^j+\eps_{ij}b^j\varphi. 
\endaligned
$$ 
Furthermore, the connection form $\zeta$ pulls-back to $F^+_g$ to become\footnote{In order to keep notation uncluttered we omit writing $\lambda^*$ for pull-backs by $\lambda$.}
$$
\zeta=\begin{pmatrix} -b_1\eta^1-b_2\eta^2 & b_1\eta^2-b_2\eta^1-\varphi\\ -b_1\eta^2+b_2\eta^1+\varphi & -b_1\eta^1-b_2\eta^2\end{pmatrix}=\begin{pmatrix} -\beta & \star \beta -\varphi \\ \varphi -\star \beta & -\beta\end{pmatrix},
$$
where $\star$ denotes the Hodge-star with respect to the orientation and metric $g$. A simple calculation now shows that the components of the Schouten tensor are
$$
S=\begin{pmatrix} K+b_{11}+b_{22} & -\frac{1}{3}(b_{12}+b_{21})\\ \frac{1}{3}(b_{12}-b_{21}) & K+b_{11}+b_{22}\end{pmatrix}. 
$$
Note that the diagonal entry of $S$ is $K-\delta \beta$ where $\delta$ denotes the co-differential with respect to the orientation and metric $g$.  

If we now apply the formula~\eqref{cartanconnpref} for the Cartan connection of the projective structure defined by ${}^{(g,\beta)}\nabla$ -- whilst setting $\xi \equiv 0$ -- we obtain
\begin{equation}\label{weylgauge}
\phi=\left(\begin{array}{ccc} \frac{2}{3}\beta & (\delta \beta-K)\eta^1+\frac{1}{3}(\star \d \beta)\eta^2 & -\frac{1}{3}(\star \d \beta)\eta^1+(\delta\beta-K)\eta^2\\ \eta^1 & -\frac{1}{3}\beta & \star \beta-\levcon\\ \eta^2 & \levcon-\star\beta &-\frac{1}{3}\beta\end{array}\right). 
\end{equation}
Denoting by $(\pi : B\to \Sigma, \theta)$ the Cartan geometry associated to the projective structure defined by ${}^{(g,\beta)}\nabla$, it follows from the uniqueness part of Cartan's construction that there exists an $\mathrm{SO}(2)$-bundle embedding $\psi : F^+_g \to B$ so that $\psi^*\theta=\phi$.

For the sake of completeness we also record
$$
\d \phi+\phi\wedge\phi=\left(\begin{array}{ccc} 0 & \hat W_1 \phi^1_0\wedge\phi^2_0 &\hat W_2 \phi^1_0\wedge\phi^2_0\\ 0 & 0 & 0\\ 0 & 0 &0 \end{array}\right),
$$
where
$$
\hat W_1\phi^1_0+\hat W_2\phi^2_0=-\star \d (K-\delta\beta)+\frac{1}{3}\d \star \d \beta-2(K-\delta\beta)\star\beta-\frac{2}{3}\beta\star\d\beta. 
$$

\section{Flexibility and rigidity of holomorphic curves}

A conformal structure $[g]$ on the oriented surface $\Sigma$ is the same as a smooth choice of a positively oriented orthonormal coframe for every point $p \in \Sigma$, well defined up to rotation and scaling; in other words, a smooth section of $F^+/\mathrm{CO}(2)$ where $\mathrm{CO}(2)=\R^+\times \mathrm{SO}(2)$ is the linear conformal group. 

Assume $\Sigma$ to be equipped with a projective structure $\mathfrak{p}$ and let $(\pi : B \to \Sigma,\theta)$ denote its associated Cartan geometry. Recall that $F^+$ is obtained as the quotient of $B$ by the normal subgroup $\R^2\rtimes\left\{\mathrm{Id}\right\}\subset G$, hence the conformal structures on $\Sigma$ are in one-to-one correspondence with the sections of $\tau : B/\left(\R^2\rtimes \mathrm{CO}(2)\right) \to \Sigma$, where $\tau$ denotes the base-point projection. By construction, the typical fibre of $\tau$ is the homogeneous space $\mathrm{GL}^+(2,\R)/\mathrm{CO}(2)$ which is diffeomorphic to the open unit disk in $\C$.   

In~\cite{MR728412,MR812312} it was shown that $\mathfrak{p}$ induces a complex structure $J$ on the space $B/\left(\R^2\rtimes \mathrm{CO}(2)\right)$, thus turning this quotient into a complex surface $Z$. The complex structure on $Z$ can be characterised in terms of the Cartan connection $\theta$ on $B$. To this end we write the structure equations of $\theta$ in complex form. 
\begin{lemma}\label{complexstruceq}
Writing
$$
\aligned
\omega_1&=\theta^1_0+\i\theta^2_0,\\
\omega_2&=(\theta^1_1-\theta^2_2)+\i\left(\theta^1_2+\theta^2_1\right),\\
\xi&=\theta^0_1+\i\theta^0_2,\\
\psi&=-\frac{1}{2}\left(3\theta^0_0+\i(\theta^1_2-\theta^2_1)\right),
\endaligned
$$
we have
\begin{equation}\label{struceqcplxform}
\aligned
\d \omega_1&=\omega_1\wedge\psi+\frac{1}{2}\ov{\omega_1}\wedge\omega_2,\\
\d \omega_2&=-\omega_1\wedge\xi+\omega_2\wedge\psi+\ov{\psi}\wedge\omega_2,\\
\d \xi&=W\ov{\omega_1}\wedge\omega_1-\frac{1}{2}\ov{\xi}\wedge\omega_2+\ov{\psi}\wedge\xi,\\
\d \psi&=-\frac{1}{2}\ov{\omega_1}\wedge\xi+\frac{1}{4}\ov{\omega_2}\wedge\omega_2+\ov{\xi}\wedge\omega_1,
\endaligned
\end{equation}
where $W=\frac{1}{2}(W_2-\i W_1)$ and $\ov{\alpha}$ denotes complex conjugation of the com\-plex-valued form $\alpha$.  
\end{lemma}
\begin{proof}
The proof is a straightforward translation of the structure equations~\eqref{struceqprojsurf} into complex form. 
\end{proof}

Using Lemma~\ref{complexstruceq} we can prove:
\begin{prp}\label{charcplxstruccart}
There exists a unique integrable almost complex structure $J$ on $Z$ such that a complex-valued $1$-form $\alpha$ on $Z$ is a $(1,\! 0)$-form for $J$ if and only if the pullback of $\alpha$ to $B$ is a linear combination of $\omega_1$ and $\omega_2$.  
\end{prp}

\begin{proof}
By definition, we have $Z=B/H$ where $H=\R^2\rtimes \mathrm{CO}(2)\subset G$. The Lie algebra $\mathfrak{h}$ of $H$ consists of matrices of the form
$$
\left(\begin{array}{rrc} -2h_4 & h_1 & h_2 \\ 0 & h_4 & h_3 \\ 0 & -h_3 & h_4\end{array}\right),
$$
where $h_1,\ldots,h_4$ are real numbers. Therefore, since the Cartan connection maps every fundamental vector field $X_v$ on $B$ to its generator $v$, the $1$-forms $\omega_1,\omega_2$ are semibasic for the projection to $Z$, that is, vanish on vector fields that are tangent to the fibres of $B \to Z$. Consequently, the  pullback to $B$ of a $1$-form on $Z$ is a linear combination of $\omega_1,\omega_2$ and their complex conjugates. We write the elements of $H$ in the following form
$$
z\rtimes r\e^{\i\phi}=\left(\begin{array}{ccc} r^{-2} & \mathrm{Re}(z) & \mathrm{Im}(z) \\ 0 & r\cos\phi & r \sin\phi \\ 0 & -r \sin\phi & r \cos \phi\end{array}\right),
$$ 
where $z \in \C$ and $r\e^{\i\phi}\in \C^*$. The equivariance of $\theta$ under the $G$-right action gives
$$
\left(R_{b\rtimes a}\right)^*\theta=(b\rtimes a)^{-1}\theta(b\rtimes a)=(-(\det a)ba^{-1}\rtimes a^{-1})\theta (b\rtimes a)
$$
which implies
\begin{equation}\label{rightactionform1}
\left(R_{\rot}\right)^*\omega_1=\frac{1}{r^3}\e^{\i\phi}\omega_1
\end{equation}
and
\begin{equation}\label{rightactionform2}
\left(R_{\rot}\right)^*\omega_2=\frac{z}{r}\e^{\i\phi}\omega_1+\e^{2\i\phi}\omega_2,
\end{equation}
thus showing that there exists a unique almost complex structure $J$ on $Z$ such that a complex-valued $1$-form $\alpha$ on $Z$ is a $(1,\! 0)$-form for $J$ if and only if the pullback of $\alpha$ to $B$ is a linear combination of $\omega_1$ and $\omega_2$. The integrability of $J$ is now a consequence of the complex form of the structure equations given in Lemma~\ref{complexstruceq} and the Newlander-Nirenberg theorem.  
\end{proof}
Using this characterisation we have~\cite[Theorem 3]{MR3144212}:
\begin{thm}\label{mainweylmetri}
Let $(\Sigma,\mathfrak{p})$ be an oriented projective surface. A conformal structure $[g]$ on $\Sigma$ is preserved by a conformal connection defining $\mathfrak{p}$ if and only if the image of $[g] : \Sigma \to Z$ is a holomorphic curve.  
\end{thm}
 
\subsection{Chern-class of the co-normal bundle}
Here we use the characterisation of the complex structure on $Z$ in terms of the Cartan connection $\theta$ to compute the degree of the normal bundle of a holomorphic curve $D\subset Z$ arising as the image of a section of $Z \to \Sigma$.  
  
\begin{lemma}\label{keyprop}
Let $(\Sigma,\mathfrak{p})$ be a closed oriented projective surface and $[g] : \Sigma \to Z$ a section with holomorphic image. Then the normal bundle of the holomorphic curve $D=[g](\Sigma)\subset Z$ has degree $2\,\chi(\Sigma)$.
\end{lemma}

\begin{proof}
We will compute the degree of the co-normal bundle of $D=[g](\Sigma)\subset Z$ by computing its first Chern-class. We let $B^{\prime}\subset B$ denote the subbundle consisting of those elements $b \in B$ whose projection to $Z$ lies in $D$. Consequently, $B^{\prime} \to D$ is a principal right $H$-bundle.

The characterisation of the complex structure on $Z$ given in Proposition~\ref{charcplxstruccart} implies that the sections of the rank $2$ vector bundle  
$$
T^{1,0}Z^*\vert_D \to D
$$
correspond to functions $\lambda=(\lambda^i) : B^{\prime} \to \C^2$ such that
$$
\left(\omega_1\;\omega_2\right)\cdot \left(\begin{array}{c} \lambda^1 \\ \lambda^2\end{array}\right)=\lambda^1\omega_1+\lambda^2\omega_2
$$
is invariant under the $H$-right action. Using~\eqref{rightactionform1} and~\eqref{rightactionform2} we see that this condition on $\lambda$ is equivalent to the equivariance of $\lambda$ with respect to the right action of $H$ on $B^{\prime}$ and the right action of $H$ on $\C^2$ induced by the representation 
$$
\chi : H \to \mathrm{GL}(2,\C), \quad z\rtimes r \e^{\i\phi} \mapsto \left(\begin{array}{cc} \frac{1}{r^3}\e^{\i\phi}& \frac{z}{r}\e^{\i\phi}\\ 0 & \e^{2\i\phi} \end{array}\right).
$$
Similarly, we see that the $(1,\! 0)$-forms on $D$ are in one-to-one correspondence with the complex-valued functions on $B^{\prime}$ that are equivariant with respect to the right action of $H$ on $B^{\prime}$ and the right action of $H$ on $\C$ induced by the representation
$$
\rho : H \to \mathrm{GL}(1,\C), \quad z\rtimes r\e^{\i\phi} \mapsto \frac{1}{r^3}\e^{\i\phi}.
$$
The representation $\rho$ is a subrepresentation of $\chi$, hence the quotient representation $\chi/\rho$ is well defined and the sections of the co-normal bundle of $D$ are therefore in one-to-one correspondence with the com\-plex-valued functions $\nu$ on $B^{\prime}$ that satisfy the equivariance condition
$$
\nu(b \cdot z \rtimes r \e^{\i\phi})=(\chi/\rho)\left((z\rtimes r \e^{\i\phi})^{-1}\right)\nu(b)=\e^{-2\i\phi}\nu(b)
$$
for all $b \in B^{\prime}$ and $z \rtimes r\e^{\i\phi} \in H$. Here we have used that the quotient representation $\chi/\rho$ is isomorphic to the complex one-dimensional representation of $H$
$$
z\rtimes r\e^{\i\phi} \mapsto \e^{2\i\phi}. 
$$
In particular, given two such complex-valued functions $\nu_1,\nu_2$ on $B^{\prime}$, we may define 
$$
\langle \nu_1,\nu_2\rangle=\nu_1\ov{\nu_2},
$$
which equips the co-normal bundle $N^* \to D$ with a Hermitian bundle metric $h$. 

We will next compute the Chern connection of $h$ and express it in terms of the Cartan connection $\theta$. This can be done most easily by further reducing the bundle $B^{\prime}\subset B$. Since $D\subset Z$ is the image of a section of $Z \to \Sigma$ and is a holomorphic curve, it follows from the characterisation of the complex structure $J$ on $Z$ given in Proposition~\ref{charcplxstruccart} that there exists a complex-valued function $f$ on $B^{\prime}$ such that
$$
\omega_2=f\omega_1. 
$$
Using the formulae~\eqref{rightactionform1} and~\eqref{rightactionform2} again, it follows that the function $f$ satisfies
$$
f(b\cdot z \rtimes r\e^{\i\phi})=r^2\left(r\e^{\i\phi}f(b)+z\right)
$$
for all $b\in B^{\prime}$ and $z\rtimes r\e^{\i\phi} \in H$. Consequently, the condition $f\equiv 0$ defines a principal right $\mathrm{CO}(2)$-subbundle $B^{\prime\prime} \to D$ on which $\omega_2$ vanishes identically. The representation $\chi/\rho$ restricts to define a representation of the subgroup $\mathrm{CO}(2)\subset H$ and therefore, the sections of the co-normal bundle of $D$ are in one-to-one correspondence with the complex-valued functions $\nu$ on $B^{\prime\prime}$ satisfying the equivariance condition
\begin{equation}\label{rightsecconorm}
\nu(b \cdot r\e^{\i\phi})=\e^{-2\i\phi}\nu(b)
\end{equation}
for all $b \in B^{\prime\prime}$ and $r\e^{\i\phi}$ in $\mathrm{CO}(2)$. Equation~\eqref{rightsecconorm} implies that infinitesimally $\nu$ must satisfy
$$
\d \nu=\nu^{(1,0)}\omega_1+\nu^{(0,1)}\ov{\omega_1}+\nu\left(\psi-\ov{\psi}\right)
$$
for unique complex-valued functions $\nu^{(1,0)}$ and $\nu^{(0,1)}$ on $B^{\prime\prime}$. A simple computation shows that the form $\psi-\ov{\psi}$ is invariant under the $\mathrm{CO}(2)$ right action, therefore it follows that the map
$$
\nabla_{\mathfrak{p}} : \Gamma(D,N^*) \to \Omega^1(D,N^*), \quad  \nu \mapsto \d \nu-\nu\left(\psi-\ov{\psi}\right) 
$$
defines a connection on the co-normal bundle of $D\subset Z$. By construction, this connection preserves $h$. As a consequence of the characterisation of the complex structure on $Z$, it follows that a section $\nu$ of the co-normal bundle $N^* \to D$ is holomorphic if and only if $\nu^{0,1}=0$. This shows that $\nabla_{\mathfrak{p}}^{0,1}=\bar\partial_{N^*}$, that is, the connection $\nabla_{\mathfrak{p}}$ must be the Chern-connection of $h$. The Chern-connection of $h$ has curvature
$$
\d\left(\ov{\psi}-\psi\right)=\frac{1}{2}\left(\omega_1\wedge\ov{\xi}-\ov{\omega_1}\wedge\xi\right)
$$
where we have used the structure equations~\eqref{struceqcplxform} and that $\omega_2\equiv 0$ on $B^{\prime\prime}$. Since we have a section $[g] : \Sigma \to Z$ whose image is a holomorphic curve, we know from Theorem~\ref{mainweylmetri} that $\mathfrak{p}$ is defined by a conformal connection. Let $g$ be any metric defining $[g]$ and denote by $F^+_g \to \Sigma$ the $\mathrm{SO}(2)$-bundle of positively oriented $g$-orthonormal coframes. By the uniqueness part of Cartan's bundle construction we must have an $\mathrm{SO}(2)$-bundle embedding $\psi : F^+_g \to B^{\prime\prime}$ covering the identity on $\Sigma\simeq D$ so that $\psi^*\theta=\phi$ where $\phi=(\phi^i_j)_{i,j=0,1,2}$ is given in~\eqref{weylgauge}. Recall that $\omega_2$ vanishes identically on $B^{\prime\prime}$ which is consistent with~\eqref{weylgauge}, since
$$
\aligned
\psi^*\omega_2&=(\phi^1_1-\phi^2_2)+\i\left(\phi^1_2+\phi^2_1\right)\\
&=-\frac{1}{3}\beta-\left(-\frac{1}{3}\beta\right)+\i\left((\star \beta-\levcon)+(\levcon-\star \beta)\right)=0. \\
\endaligned
$$
Therefore, by using~\eqref{weylgauge}, we see that the curvature of $\nabla_{\mathfrak{p}}$ is given by
$$
\d\left(\ov{\psi}-\psi\right)=2\i \left(K-\delta\beta\right)\d\mu.
$$
where $\d \mu=\eta^1\wedge\eta^2$ denotes the area form of $g$. 
Concluding, we have shown that the first Chern-class $c_1(N^*) \in H^2(D,\mathbb{Z})$ of $N^*\to D$ is given by
$$
c_1(N^*)=\left[\frac{1}{\pi}(\delta\beta-K)\d\mu\right]=\left[-\frac{K}{\pi}\d\mu\right].
$$
Hence the degree of $N^*\to D$ is 
$$
\mathrm{deg}(N^*)=\int_{D}c_1(N^*)=-2\chi(\Sigma),
$$
by the Gauss-Bonnet theorem. It follows that the normal bundle $N \to D$ has degree $2\chi(\Sigma)$. 
\end{proof}

\subsection{Rigidity of holomorphic curves}

We are now ready to prove the following rigidity result. 
\begin{prp}\label{main}
Let $(\Sigma,\mathfrak{p})$ be a closed oriented projective surface satisfying $\chi(\Sigma)<0$. Then there exists at most one section $[g] : \Sigma \to Z$ whose image is a holomorphic curve.  
\end{prp}

\begin{proof}
Let $[g] : \Sigma \to Z$ be a section whose image $D=[g](\Sigma)$ is a holomorphic curve. Since $D$ is an effective divisor, the divisor/line bundle correspondence yields a holomorphic line bundle $L \to Z$ and a holomorphic section $\sigma: Z \to L$ so that $\sigma$ vanishes precisely on $D$. Recall that the fibre of $Z \to \Sigma$ is the open unit disk and hence contractible. It follows that the projection to $Z \to \Sigma$ induces an isomorphism 
$
\mathbb{Z}\simeq H^2(\Sigma,\mathbb{Z})\simeq H^2(Z,\mathbb{Z}). 
$
In particular, every smooth section of $Z \to \Sigma$ induces and isomorphism $H^2(Z,\mathbb{Z})\simeq H^2(\Sigma,\mathbb{Z})$ on the second integral cohomology groups and any two such isomorphisms agree. Keeping this in mind we now suppose that $[\hat{g}] : \Sigma \to Z$ is another section whose image is a holomorphic curve. Using the functoriality of the first Chern class we compute the degree of $L \to Z$ restricted to $D^{\prime}=[\hat{g}](\Sigma)$
$$
\mathrm{deg}\left(L\vert_{D^\prime}\right)=\int_{D^{\prime}}c_1(L)=\int_\Sigma[g]^*\left(c_1(L)\right)=\int_{D}c_1(L)=\mathrm{deg}\left(L\vert_{D}\right)
$$
where $c_1(L) \in H^2(Z,\Z)$ denotes the first Chern-class of the line bundle $L \to Z$. Using the first adjunction formula 
$$
N(D)\simeq L\vert_D
$$
and Lemma~\ref{keyprop} yields
$$
\mathrm{deg}\left(L\vert_{D^\prime}\right)=\mathrm{deg}\left(L\vert_{D}\right)=\mathrm{deg}\left(N(D)\right)<0. 
$$
Since $L\vert_{D^{\prime}}\to D^{\prime}$ has negative degree, it follows that its only holomorphic section is the zero section. Consequently, $\sigma$ vanishes identically on $D^{\prime}$. Since $\sigma$ vanishes precisely on $D$ we obtain the desired uniqueness $D=D^{\prime}$.   
\end{proof}
Combining Theorem~\ref{mainweylmetri} and Proposition~\ref{main} we get:
\begin{thm}\label{sameconformal}
Let $\Sigma$ be a closed oriented surface $\Sigma$ with $\chi(\Sigma)<0$. Then the map 
$$
\mathfrak{W}(\Sigma) \to \mathfrak{P}(\Sigma), \quad ([g],\nabla)\mapsto \mathfrak{p}(\nabla),
$$
which sends a Weyl structure to the projective equivalence class of its conformal connection, is injective.
\end{thm}
\begin{proof}
Let $([g],\nabla)$ and $([\hat{g}],\nabla^{\prime})$ be Weyl structures on $\Sigma$ having projectively equivalent conformal connections. Let $\mathfrak{p}$ be the projective structure defined by $\nabla$ (or $\nabla^{\prime})$. By Theorem~\ref{mainweylmetri} both $[g] : \Sigma \to Z$ and $[\hat{g}] : \Sigma \to Z$ have holomorphic image, with respect to the complex structure on $Z$ induced by $\mathfrak{p}$, and hence must agree by Proposition~\ref{main}. Since $\nabla$ and $\nabla^{\prime}$ are projectively equivalent, it follows that we may write
$$
\nabla+\iota(\alpha)=\nabla+\alpha\otimes \mathrm{Id}+\mathrm{Id}\otimes\alpha=\nabla^{\prime}
$$
for some $1$-form $\alpha$ on $\Sigma$. Since $\nabla$ and $\nabla^{\prime}$ are conformal connections for the same conformal structure $[g]$, there must exist $1$-forms $\beta$ and $\hat{\beta}$ on $\Sigma$ so that
$$
{}^g\nabla+g\otimes \beta^{\sharp}-\iota(\beta-\alpha)={}^g\nabla+g\otimes \hat{\beta}^{\sharp}-\iota(\hat{\beta}).
$$
Hence we have
$$
0=g\otimes \left(\beta^{\sharp}-\hat{\beta}^{\sharp}\right)+\iota(\alpha+\hat{\beta}-\beta).
$$
Writing $\gamma=\alpha+\hat{\beta}-\beta$ as well as $X=\beta^{\sharp}-\hat{\beta}^{\sharp}$ and taking the trace gives $3\,\gamma=\hat{\beta}-\beta=-X^{\flat}$. We thus have
$$
0=g\otimes X-\frac{1}{3}\,\iota(X^{\flat}).
$$
Contracting this last equation with the dual metric $g^\#$ implies $X=0$. It follows that $\alpha$ vanishes too and hence $\nabla=\nabla^{\prime}$ as claimed.
\end{proof}
Since exact Weyl structures correspond to Riemannian metrics up to constant rescaling, we immediately obtain~\cite{MR1796527}: 
\begin{cor}
A Riemannian metric $g$ on a closed oriented surface $\Sigma$ satisfying $\chi(\Sigma)<0$ is uniquely determined -- up to constant rescaling -- by its unparametrised geodesics.
\end{cor}
\begin{rmk}
The first (non-compact) examples of non-trivial pairs of projectively equivalent Riemannian metrics, that is, metrics sharing the same unparametrised geodesics, go back to Beltrami~\cite{beltrami}.
\end{rmk}
\begin{rmk}\label{extor}
Clearly, pairs of distinct flat tori (after pulling back the metrics to $T^2=S^1\times S^1$) yield pairs of Riemannian metrics on the $2$-torus that are (generically) not constant rescalings of each other, but have the same Levi-Civita connection. This fact together with Example~\ref{flatex} shows that the assumption $\chi(\Sigma)<0$ in Theorem~\ref{sameconformal} is optimal.  
\end{rmk}

\subsection{Conformal connections on the 2-sphere}

As an immediate by-product of the proof of Theorem~\ref{sameconformal} we see that a conformal connection on a closed oriented surface $\Sigma$ with $\chi(\Sigma)<0$ preserves precisely one conformal structure. By Remark~\ref{extor}, this is false in general on the $2$-torus. It is therefore natural to ask if a conformal connection on the $2$-sphere can preserve more than one conformal structure. We will show next that this is not the case. Let therefore $(g,\beta)$ and $(h,\alpha)$ on $S^2$ be such that the associated conformal connections agree
$$
{}^{(g,\beta)}\nabla={}^{(h,\alpha)}\nabla=\nabla.
$$
Fix an orientation on $S^2$ and let $\lambda : F^+_g \to S^2$ denote the $\mathrm{SO}(2)$-bundle of positively oriented $g$-orthonormal coframes with  co\-fram\-ing $(\eta^1,\eta^2,\levcon)$ as described in~\S\ref{confconsec}. Write $\lambda^*h=h_{ij}\eta^i\otimes\eta^j$ for unique real-valued functions $h_{ij}=h_{ji}$ on $F^+_g$ and $\lambda^*\beta=b_i\eta^i$ as well as $\lambda^*\alpha=a_i\eta^i$ for unique real-valued functions $a_i,b_i$ on $F^+_g$. Recall from ~\S\ref{confconsec} that on $F^+_g$ the connection $1$-form $\zeta=(\zeta^i_j)$ of $\nabla$ takes the form
\begin{equation}\label{kappaagain}
\zeta=\begin{pmatrix} -b_1\eta^1-b_2\eta^2 & b_1\eta^2-b_2\eta^1-\varphi\\ -b_1\eta^2+b_2\eta^1+\varphi & -b_1\eta^1-b_2\eta^2\end{pmatrix}.
\end{equation}
By assumption, we have
$$
\nabla h=2\alpha\otimes h.  
$$
On $F^+_g$ this condition translates to
$$
\d h_{ij}=h_{k j}\zeta^{k}_i+h_{ik}\zeta^{k}_j+2a_{k}h_{ij}\eta^{k}. 
$$
Hence using \eqref{kappaagain} we obtain
$$
\aligned
\d (h_{11}-h_{22})&=2h_{11}\zeta^1_1-2h_{12}\zeta^1_2+2h_{21}\zeta^2_1-2h_{22}\zeta^2_2+2a_{k}\eta^k(h_{11}-h_{22})\\
&=2\left[(a_k-b_k)(h_{11}-h_{22})\eta^k+2h_{12}\zeta^2_1\right],\\
\d h_{12}&=h_{12}\zeta^1_1+h_{22}\zeta^2_1+h_{11}\zeta^1_2+h_{12}\zeta^2_2+2h_{12}a_k\eta^k\\
&=2(a_k-b_k)h_{12}\eta^k-(h_{11}-h_{22})\zeta^2_1.
\endaligned
$$
Writing 
$$
f=(h_{11}-h_{22})^2+4(h_{12})^2,
$$
we get
\begin{equation}\label{derivconfdif}
\aligned
\d f=&\,4(h_{11}-h_{22})\left[(a_k-b_k)(h_{11}-h_{22})\eta^k+2h_{12}\zeta^2_1\right]+8h_{12}\cdot\\
&\cdot\left[2(a_k-b_k)h_{12}\eta^k-(h_{11}-h_{22})\zeta^2_1\right]\\
=&\,4(h_{11}-h_{22})^2(a_k-b_k)\eta^k+16(h_{12})^2(a_k-b_k)\eta^k\\
=&\,4f(a_k-b_k)\eta^k. 
\endaligned
\end{equation}
In particular, the function $f$ on $F^+_g$ is constant along the $\lambda$-fibres and hence the pullback of a unique function on $S^2$ which we will also denote by $f$. The Ricci curvature of $\nabla$ is 
$$
\mathrm{Ric}(\nabla)=(K_g-\delta_g \beta)g-2\d \beta=(K_{h}-\delta_{h}\alpha)h-2\d\alpha. 
$$
It follows that the metrics $g$ and $h$ are conformal on the non-empty open subset $\Sigma^{\prime}\subset S^2$ where the symmetric part of the Ricci curvature is positive definite. Since $\d \alpha=\d \beta$ and $H^1(S^2)=0$, we must have that $\alpha-\beta=\d u$ for some real-valued function $u$ on $S^2$. Consequently, it follows from~\eqref{confresc} that after possibly conformally rescaling $h$ we can assume $\alpha=\beta$ (and hence $a_i=b_i$) without loss of generality. Therefore, \eqref{derivconfdif} implies that $f$ is constant. By construction, the function $f$ vanishes precisely at the points where $h$ is conformal to $g$. Since we already know that $f$ vanishes on the open subset $\Sigma^{\prime}$ it must vanish on all of $S^2$. We have thus proved:
\begin{prp}\label{ones2}
A conformal connection on the $2$-sphere preserves precisely one conformal structure. 
\end{prp} 

Combining Lemma~\ref{keyprop}, Theorem~\ref{mainweylmetri} and Proposition~\ref{ones2} with Kodaira's deformation theorem~\cite{MR0133841}, we obtain the following result about the deformation space of a conformal connection on the $2$-sphere $S^2$. 
\begin{thm}\label{deform}
Every conformal connection on the $2$-sphere lies in a complex $5$-manifold of conformal connections, all of which share the same unparametrised geodesics. 
\end{thm}
\begin{rmk}
Recall that Kodaira's theorem states that if $Y\subset Z$ is an embedded compact complex submanifold of some complex manifold $Z$ and satisfies $H^1(Y,\mathcal{O}(N))=0$, then $Y$ belongs to a locally complete family $\left\{Y_x \,|\, x\in X\right\}$ of compact complex submanifolds of $Z$, where $X$ is a complex manifold. Furthermore, there is a canonical isomorphism $T_xX\simeq H^0(Y_x,\mathcal{O}(N))$.   
\end{rmk}
\begin{proof}[Proof of Theorem~\ref{deform}]
Let $\nabla$ be a conformal connection on the oriented $2$-sphere defining the projective structure $\mathfrak{p}$. Let $[g] : S^2 \to Z$ be the conformal structure that is preserved by $\nabla$, then Theorem~\ref{mainweylmetri} implies that $Y=[g](S^2)\subset Z$ is a holomorphic curve biholomorphic to $\mathbb{CP}^1$. By Lemma~\ref{main} the normal bundle $N$ of $Y\subset Z$ has degree $4$ and hence we have (by standard results)
$$
\dim H^1(\mathbb{CP}^1,\mathcal{O}(4))=0, \quad \text{and}\quad \dim H^0(\mathbb{CP}^1,\mathcal{O}(4))=5.
$$
Consequently, Kodaira's theorem applies and $Y$ belongs to a locally complete family $\left\{Y_x \,|\, x\in X\right\}$ of holomorphic curves of $Z$, where $X$ is a complex $5$-manifold. A holomorphic curve in the family $X$ that is sufficiently close to $Y$ will again be the image of a section of $Z \to S^2$ and hence yields a conformal structure $[g^{\prime}]$ on $S^2$ that is preserved by a conformal connection $\nabla^{\prime}$ defining $\mathfrak{p}$. Since by Proposition~\ref{ones2} a conformal connection on $S^2$ preserves precisely one conformal structure, the claim follows. \end{proof}

\begin{rmk} 
In~\cite[Corollary 2]{MR3144212} it was shown that the conformal connections on $S^2$ whose (unparametrised) geodesics are the great circles are in one-to-one correspondence with the smooth quadrics in $\mathbb{CP}^2$ without real points. The space of smooth quadrics in $\mathbb{CP}^2$ is the complex $5$-dimensional space $\rm PSL(3,\C)/\rm PSL (2,\C)$, with the smooth quadrics without real points being an open submanifold thereof. Thus, the space of smooth quadrics without real points is complex five-dimensional, which is in agreement with Theorem~\ref{deform}.
\end{rmk}

\begin{rmk}
Inspired by the work of Hitchin~\cite{MR699802} (treating the case $n=2$) and Bryant~\cite{MR1141197} (treating the case $n=3$), it was shown in~\cite{MR2240424} that the deformation space $\mathcal{M}^{n+1}$ of a holomorphically embedded rational curve with self-intersection number $n\geq 2$ in a complex surface $Z$ comes canonically equipped with a holomorphic $\mathrm{GL}(2)$-structure, which is a (holomorphically varying) identification of every holomorphic tangent space of $\mathcal{M}$ with the space of homogeneous polynomials of degree $n$ in two complex variables. Therefore, every conformal connection on the $2$-sphere gives rise to a complex $5$-manifold $\mathcal{M}$ carrying a holomorphic $\mathrm{GL}(2)$-structure. 
\end{rmk}

\begin{rmk}
It is an interesting problem to classify the pairs of Weyl structures on the $2$-torus having projectively equivalent conformal connections. The Riemannian case was treated in~\cite{MR2053913}. 
\end{rmk}

\providecommand{\bysame}{\leavevmode\hbox to3em{\hrulefill}\thinspace}
\providecommand{\noopsort}[1]{}
\providecommand{\mr}[1]{\href{http://www.ams.org/mathscinet-getitem?mr=#1}{MR~#1}}
\providecommand{\zbl}[1]{\href{http://www.zentralblatt-math.org/zmath/en/search/?q=an:#1}{Zbl~#1}}
\providecommand{\jfm}[1]{\href{http://www.emis.de/cgi-bin/JFM-item?#1}{JFM~#1}}
\providecommand{\arxiv}[1]{\href{http://www.arxiv.org/abs/#1}{arXiv~#1}}
\providecommand{\doi}[1]{\href{http://dx.doi.org/#1}{doi$>\!$}}
\providecommand{\MR}{\relax\ifhmode\unskip\space\fi MR }
\providecommand{\MRhref}[2]{%
  \href{http://www.ams.org/mathscinet-getitem?mr=#1}{#2}
}
\providecommand{\href}[2]{#2}

\flushleft
\sc
\footnotesize
\bigskip
Department of Mathematics,
ETH Z\"urich,
8092 Z\"urich, Switzerland, and,\\
Department of Mathematics,
University of Fribourg,
1700 Fribourg, Switzerland\\
\tt
\href{mailto:mettler@math.ch}{mettler@math.ch}

\end{document}